\documentclass[11pt]{article}
\usepackage{amsmath}
\usepackage{amssymb}
\usepackage{amsthm}
\usepackage[usenames]{color}
\usepackage{amscd}
\usepackage{indentfirst}

\usepackage[colorlinks=true,linkcolor=blue,filecolor=red,
citecolor=webgreen]{hyperref}
\definecolor{webgreen}{rgb}{0,.5,0}

\def\N{{\Bbb N}}

\def\C{{\Bbb C}}

\def\1{{\bf 1}}

\newtheorem{theorem}{Theorem}[section]
\newtheorem{lemma}[theorem]{Lemma}

\begin{document}

\title{{\bf Counting $r$-tuples of positive integers with $k$-wise relatively prime components}}
\author{L\'aszl\'o T\'oth}
\date{}
\maketitle

\centerline{Journal of Number Theory {\bf 166} (2016), 105--116}

\begin{abstract} Let $r\ge k\ge 2$ be fixed positive integers. Let $\varrho_{r,k}$ denote the characteristic
function of the set of $r$-tuples
of positive integers with $k$-wise relatively prime components, that
is any $k$ of them are relatively prime. We use the convolution
method to establish an asymptotic formula for the sum
$\sum_{n_1,\ldots,n_r\le x} \varrho_{r,k}(n_1,\ldots,n_r)$ by
elementary arguments. Our result improves the error term obtained by
J.~Hu \cite{Hu2013}.
\end{abstract}



\section{Introduction}

Let $r\ge k\ge 2$ be fixed positive integers. The positive integers
$n_1,\ldots,n_r$ are called $k$-wise relatively prime if any $k$ of
them are relatively prime, that is $\gcd(n_{i_1},\ldots,n_{i_k})=1$
for every $1\le i_1<\ldots<i_k\le r$. In particular, in the case
$k=2$ the integers are pairwise relatively prime and for $k=r$ they
are mutually relatively prime.

Let ${\cal{S}}_{r,k}$ denote the set of $r$-tuples of positive
integers with $k$-wise relatively prime components and let
$\varrho_{r,k}$ stand for its characteristic function. What is the
asymptotic density
\begin{equation*}
d_{r,k}= \lim_{x\to \infty} \frac1{x^r} \sum_{n_1,\ldots,n_r\le x}
\varrho_{r,k}(n_1,\ldots,n_r)
\end{equation*}
of the set ${\cal{S}}_{r,k}$? Heuristically, the probability that a
positive integer is divisible by a fixed prime $p$ is $1/p$, hence
the probability that given $r$ positive integers exactly $j$ of them
are divisible by $p$ is
\begin{equation*}
\binom{r}{j}\frac1{p^j} \left(1-\frac1{p}\right)^{r-j}
\end{equation*}
and the probability that they are $k$-wise relatively prime is
\begin{equation} \label{P_rk}
P_{r,k}= \prod_p \sum_{j=0}^{k-1} \binom{r}{j} \frac1{p^j}
\left(1-\frac1{p}\right)^{r-j}.
\end{equation}

In the case $k=2$ the above heuristic argumentation is given in
\cite[p.\ 55]{Sch2009} and one has
\begin{equation} \label{dens_k2}
P_{r,2}= \prod_p \left(1-\frac1{p} \right)^{r-1} \left(1+\frac{r-1}{p}\right).
\end{equation}

Note that for every $r\ge k\ge 2$,
\begin{equation*}
c \prod_{p>r-1} \left(1-\frac{(r-1)^2}{p^2}\right) \le
 P_{r,2} \le P_{r,k} \le P_{r,r} = \prod_p \left(1-\frac1{p^r}\right),
\end{equation*}
with some constant $c>0$ (depending on $r$), hence the infinite product \eqref{P_rk} converges. Some approximate values of $P_{r,k}$
are shown by the next Table.

\[
\vbox{\offinterlineskip \hrule \hrule \halign{ \strut \vrule \vrule \hfill $\ # \
$ \hfill & \vrule \vrule \hfill $\ # \ $ \hfill & \vrule
\hfill $\ # \ $ \hfill & \vrule \hfill $\ # \ $ \hfill & \vrule \hfill $\ # \ $ \hfill
& \vrule \hfill $\ # \ $ \hfill & \vrule
 \hfill $\ # \ $  \hfill & \vrule \hfill $\ # \ $ \hfill
\vrule \vrule \cr
 P_{r,k} & k=2 & k=3 & k=4 & k=5 & k=6 & k=7 & k=8\ \cr \noalign{\hrule \hrule}
 r=2 & 0.607 &  &  &  &  & & \ \cr \noalign{\hrule}
 r=3 & 0.286 & 0.831 &  & & & &  \ \cr \noalign{\hrule}
 r=4 & 0.114 & 0.584 & 0.923 &  & & &  \ \cr \noalign{\hrule}
 r=5 & 0.040 & 0.357 & 0.768 & 0.964 & & &  \ \cr \noalign{\hrule}
 r=6 & 0.013 & 0.195 & 0.576 & 0.873  & 0.982 & & \ \cr \noalign{\hrule}
 r=7 & 0.004 & 0.097 & 0.394 & 0.734 & 0.930 & 0.991 & \ \cr \noalign{\hrule}
 r=8 & 0.001 & 0.045 & 0.247 & 0.573 & 0.837 & 0.962 & 0.995 \ \cr} \hrule \hrule}
\]
\centerline{Table. Approximate values of $P_{r,k}$ for $2\le k\le
r\le 8$}

\medskip

If $k=r$, then it is well known that $d_{r,r}=P_{r,r}=1/\zeta(r)$ is
the correct value of the corresponding asymptotic density. The case
$k=2$ was treated by the author \cite{Tot2002} proving by an
inductive approach that
\begin{equation} \label{as_Toth}
\sum_{n_1,\ldots,n_r\le x} \varrho_{r,2}(n_1,\ldots,n_r) = d_{r,2}
x^r + O\left(x^{r-1} (\log x)^{r-1}\right),
\end{equation}
where $d_{r,2}=P_{r,2}$ is given by \eqref{dens_k2}. Here and throughout the paper the $O$ ($\ll$) notation
is used in the usual way, the implied constants depend only on $r$.

The value \eqref{dens_k2} was also deduced
by J.-Y.~Cai, E.~Bach \cite[Th.\ 3.3]{CaiBac2003} using
probabilistic arguments. P.~Moree \cite[Th.\ 2]{Mor2014} proved
\eqref{as_Toth} in the case $r=3$ using a different approach. J.~Hu
\cite{Hu2013,Hu2014} proved that $d_{r,k}=P_{r,k}$ for every $r\ge
k\ge 2$. In fact, by generalizing the method of \cite{Tot2002} it
was shown in \cite{Hu2013} that
\begin{equation} \label{as_Hu}
\sum_{n_1,\ldots,n_r\le x} \varrho_{r,k}(n_1,\ldots,n_r) = P_{r,k}
x^r + O\left(x^{r-1} (\log x)^{\delta_{r,k}}\right),
\end{equation}
where $\delta_{r,k}=\max \left\{\binom{r-1}{j}: 1\le j\le k-1
\right\}$. For $k=2$ the asymptotic formula \eqref{as_Hu} reduces to
\eqref{as_Toth}. We remark that the  asymptotic density $d_{r,2}$
was obtained by the author \cite[Sect.\ 7.2]{Tot2014} by applying
the generalized Wintner theorem due to N.~Ushiroya \cite{Ush2012}.

Similar questions were investigated in some other recent papers. J.~Hu \cite{Hu2014g}
and J.~A.~de Reyna, R.~Heyman \cite{RH2015} considered modified pairwise
coprimality conditions and by using certain graph representations
they obtained asymptotic formulas similar to \eqref{as_Hu}.
Probabilistic aspects of pairwise coprimality were investigated by
J.~L.~Fern\'andez, P.~Fern\'andez \cite{FF2015}. For example, it is
proved in \cite{FF2015} that the random variable counting the number
of coprime pairs in a random sample of length $r$, drawn from
$\{1,2,\ldots,n\}$, is asymptotically normal as $r$ tends to
infinity and $n\ge 2$ is allowed to vary with $r$. X.~Guo,
H.~Xiangqian, X.~Liu \cite{GHX2014} computed the asymptotic density
of the set of $n$-tuples of $k$-wise relatively prime polynomials
over a finite field.

It is the goal of the present paper to use a method, which differs
from all approaches mentioned above, in order to establish the
asymptotic formula \eqref{as_Hu} with a better error term. More
exactly, we take into account that the function
$\varrho_{r,k}(n_1,\ldots,n_r)$ is multiplicative, viewed as an
arithmetic function of $r$ variables. Therefore, its multiple
Dirichlet series can be expressed as an Euler product and an
explicit formula can be given for it. See the survey paper of the
author \cite{Tot2014} for basic properties of multiplicative
functions of several variables. Then we use the convolution method
to obtain the desired asymptotic formula by elementary arguments.

\section{Main results}

We use the notation $n=\prod_p p^{\nu_p(n)}$ for the prime power
factorization of $n\in \N$, the product being over the primes $p$,
where all but a finite number of the exponents $\nu_p(n)$ are zero.
Furthermore, let $e_j(x_1,\dots,x_r)=\sum_{1\le i_1<\ldots<i_j\le r}
x_{i_1}\cdots x_{i_j}$ denote the elementary symmetric polynomials
in $x_1,\ldots,x_r$ of degree $j$ ($j\ge 0$). By convention,
$e_0(x_1,\ldots,x_r)=1$.

As mentioned in the Introduction, the function $\varrho_{r,k}$ is
multiplicative, which means that
\begin{equation*}
\varrho_{r,k}(m_1n_1,\ldots,m_rn_r)=\varrho_{r,k}(m_1,\ldots,m_r)
\varrho_{r,k}(n_1,\ldots,n_r),
\end{equation*}
provided that $\gcd(m_1\cdots m_r,n_1\cdots n_r)=1$. Hence we have
\begin{equation*}
\varrho_{r,k}(n_1,\ldots,n_r)= \prod_p \varrho_{r,k}
(p^{\nu_p(n_1)}, \ldots,p^{\nu_p(n_r)})
\end{equation*}
for every $n_1,\ldots,n_r$. Also, for every $\nu_1,\ldots,\nu_r\ge
0$,
\begin{equation} \label{varrho_val}
\varrho_{r,k}(p^{\nu_1}, \ldots,p^{\nu_r})= \begin{cases} 1, &
\text{if there are at most $k-1$ values $\nu_i\ge 1$}, \\ 0, &
\text{otherwise.}
\end{cases}
\end{equation}

For the multiple Dirichlet series of the function $\varrho_{r,k}$ we
have the next result:

\begin{theorem} \label{Th_Dir_ser} Let $r\ge k\ge 2$ and let $s_i\in \C$ {\rm ($1\le i\le r$)}. If $\Re s_i>1$ {\rm ($1\le i\le r$)}, then
\begin{equation*}
\sum_{n_1,\ldots,n_r=1}^{\infty}
\frac{\varrho_{r,k}(n_1,\ldots,n_r)}{n_1^{s_1}\cdots n_r^{s_r}}=
\zeta(s_1)\cdots \zeta(s_r) D_{r,k}(s_1,\ldots,s_r),
\end{equation*}
where
\begin{equation*}
D_{r,k}(s_1,\ldots,s_r)= \prod_p \left( 1- \sum_{j=k}^r (-1)^{j-k}
\binom{j-1}{k-1} e_j(p^{s_1},\ldots,p^{s_r})\right)
\end{equation*}
is absolutely convergent if and only if $\Re (s_{i_1}+\cdots+
s_{i_j})>1$ for every $1\le i_1< \ldots < i_j\le r$ with $k\le j\le
r$.
\end{theorem}

In the case $k=2$, Theorem \ref{Th_Dir_ser} was deduced by the
author \cite[Sect.\ 5.1]{Tot2014}, based on an identity concerning a
generalization of the Busche-Ramanujan identity. See \cite[Eq.\
(4.2)]{Tot2013}.

We prove the following asymptotic formula:

\begin{theorem} \label{Th_main} If $r\ge k\ge 2$, then
\begin{equation*}
\sum_{n_1,\ldots,n_r\le x} \varrho_{r,k}(n_1,\ldots,n_r) = A_{r,k}
x^r + O\left(R_{r,k}(x)\right),
\end{equation*}
where
\begin{equation} \label{A_rk}
A_{r,k}=\prod_p \left(1-\sum_{j=k}^r (-1)^{j-k}
\binom{r}{j}\binom{j-1}{k-1}\frac1{p^j}\right)
\end{equation}
and
\begin{equation} \label{R_rk}
R_{r,k}(x) = \begin{cases} x^{r-1}, & \text{ if } \ r\ge k\ge 3, \\
x^{r-1} (\log x)^{r-1}, & \text{ if } \ r\ge k=2.
\end{cases}
\end{equation}
\end{theorem}

For $k\ge 3$ the error term $R_{r,k}(x)$ is better than in
\eqref{as_Hu}, obtained by J.~Hu \cite{Hu2013}. Note also that
$A_{r,k}=P_{r,k}$, given by \eqref{P_rk}, which follows by some
simple properties of the binomial coefficients.

\section{Preliminaries}

Consider the polynomial
\begin{equation} \label{polinom}
f(x) =\prod_{j=1}^r (x-x_j) = \sum_{j=0}^r (-1)^j e_j(x_1,\dots,x_r)
x^{r-j}.
\end{equation}

We will use that its $m$-th derivative is
\begin{equation} \label{deriv_sum}
f^{(m)}(x)= m! \sum_{j=0}^r (-1)^j \binom{r-j}{m} e_j(x_1,\dots,x_r)
x^{r-j-m} \quad (0\le m \le r),
\end{equation}
and on the other hand
\begin{equation} \label{deriv_prod}
f^{(m)}(x)= m! \sum_{1\le i_1<\ldots < i_m\le r}
\prod_{\substack{j=1\\ j\ne i_1,\ldots,i_m}}^r (x-x_j) \quad (0\le m
\le r).
\end{equation}

We also need the following auxiliary results:

\begin{lemma} \label{lemma_prod_id} If $a_j\in \C$ {\rm ($1\le j\le r$)}, then
\begin{equation*}
\prod_{j=1}^r a_j = \sum_{\ell=0}^r (-1)^{\ell} \sum_{1\le
i_1<\ldots<i_{\ell}\le r} (1-a_{i_1})\cdots (1-a_{i_{\ell}}),
\end{equation*}
where the term for $\ell=0$ is considered to be $1$.
\end{lemma}

\begin{proof} Follows from \eqref{polinom} by putting $x=1$ and $x_j=1-a_j$ ($1\le j\le r$).
\end{proof}

\begin{lemma} \label{lemma_pol_id} We have the polynomial identity
\begin{equation} \label{pol_id}
\sum_{j=0}^{k-1} \sum_{1\le i_1<\ldots <i_j\le r} x_{i_1}\cdots
x_{i_j} \prod_{\substack{\ell=1\\ \ell \ne i_1,\ldots,i_j}}^r
\left(1-x_{\ell} \right)= 1- \sum_{j=k}^r (-1)^{j-k}
\binom{j-1}{k-1} e_j(x_1,\ldots,x_r),
\end{equation}
where on the left hand side the term for $j=0$ is considered to be
$\prod_{\ell=1}^r \left(1-x_{\ell} \right)$.
\end{lemma}

Note that the left hand side of \eqref{pol_id} is a symmetric
polynomial in $x_1,\ldots,x_r$ and the right hand side shows how it
can be written as a polynomial of the elementary symmetric
polynomials.

\begin{proof} By using Lemma \ref{lemma_prod_id},
\begin{equation*}
S:= \sum_{j=0}^{k-1} \sum_{1\le i_1<\ldots <i_j\le r} x_{i_1}\cdots
x_{i_j} \prod_{\substack{\ell=1\\ \ell \ne i_1,\ldots,i_j}}^r
\left(1-x_{\ell} \right)
\end{equation*}
\begin{equation*}
= \sum_{j=0}^{k-1} \sum_{1\le i_1<\ldots <i_j\le r} \sum_{m=0}^r
(-1)^m \sum_{1\le t_1<\ldots<t_m\le r} (1-x_{i_{t_1}})\cdots
(1-x_{i_{t_m}}) \prod_{\substack{\ell=1\\ \ell\ne i_1,\ldots,i_j}}^r
\left(1-x_{\ell} \right).
\end{equation*}

In the last product a number of $j$ factors are missing from the
factors $1-x_1,\ldots,1-x_r$. But a number of $m$ factors from the
missing ones are present in front of the last product. Hence the
number of missing factors is $q=j-m$, where $0\le q\le j$. We obtain
\begin{equation*}
S= \sum_{j=0}^{k-1} \sum_{1\le i_1<\ldots <i_j\le r} \sum_{q=0}^j
(-1)^{j-q} \sum_{1\le u_1<\ldots<u_q\le j} \prod_{\substack{\ell=1\\
\ell\ne i_{u_1},\ldots,i_{u_q}}}^r \left(1-x_{\ell} \right)
\end{equation*}
\begin{equation*}
= \sum_{j=0}^{k-1} (-1)^j \sum_{q=0}^j (-1)^q \sum_{1\le
u_1<\ldots<u_q\le j} \sum_{1\le i_1<\ldots
<i_j\le r} \prod_{\substack{\ell=1\\
\ell\ne i_{u_1},\ldots,i_{u_q}}}^r \left(1-x_{\ell} \right),
\end{equation*}
by regrouping the terms. Here for fixed $u_1,\ldots,u_q$ the values
$a_1=i_{u_1},\ldots,a_q=i_{u_q}$ are also fixed and the other $r-q$
values of $i_1,\ldots,i_j$ can be selected in $\binom{r-q}{j-q}$
ways. Therefore,
\begin{equation*}
S = \sum_{j=0}^{k-1} (-1)^j \sum_{q=0}^j (-1)^q \binom{r-q}{j-q}
\sum_{1\le a_1<\ldots<a_q\le r} \prod_{\substack{\ell=1\\
\ell\ne a_1,\ldots,a_q}}^r \left(1-x_{\ell} \right)
\end{equation*}
\begin{equation*}
= \sum_{q=0}^{k-1} (-1)^q \sum_{1\le a_1<\ldots<a_q\le r} \left( \prod_{\substack{\ell=1\\
\ell\ne a_1,\ldots,a_q}}^r \left(1-x_{\ell} \right)\right)
\sum_{j=q}^{k-1} (-1)^j \binom{r-q}{j-q},
\end{equation*}
where the last sum is $(-1)^{k-1}\binom{r-q-1}{k-q-1}$ and we deduce
\begin{equation*}
S = \sum_{q=0}^{k-1} (-1)^{q+k-1} \binom{r-q-1}{k-q-1} \sum_{1\le
a_1<\ldots<a_q\le r} \prod_{\substack{\ell=1\\ \ell\ne
a_1,\ldots,a_q}}^r \left(1-x_{\ell} \right).
\end{equation*}

Now using the identities \eqref{deriv_prod} and \eqref{deriv_sum}
for $x=1$ we conclude
\begin{equation*}
S = \sum_{q=0}^{k-1} (-1)^{q+k-1} \binom{r-q-1}{k-q-1} \frac1{q!}
f^{(q)}(1)
\end{equation*}
\begin{equation*}
= \sum_{q=0}^{k-1} (-1)^{q+k-1} \binom{r-q-1}{k-q-1} \sum_{j=0}^r
(-1)^j \binom{r-j}{q} e_j(x_1,\ldots,x_r)
\end{equation*}
\begin{equation*}
= \sum_{j=0}^r (-1)^{j-k+1} e_j(x_1,\ldots,x_r) \sum_{q=0}^{k-1}
(-1)^q \binom{r-q-1}{(k-1)-q} \binom{r-j}{q},
\end{equation*}
where the last sum is $\binom{j-1}{k-1}$ by the Vandermonde
identity. Hence
\begin{equation*}
S = \sum_{j=0}^r (-1)^{j-k+1} e_j(x_1,\ldots,x_r)\binom{j-1}{k-1}=
1- \sum_{j=k}^r (-1)^{j-k} \binom{j-1}{k-1} e_j(x_1,\ldots,x_r),
\end{equation*}
which completes the proof.
\end{proof}

\section{Proofs}

\begin{proof}[Proof of Theorem {\rm \ref{Th_Dir_ser}}]
The function $(n_1,\ldots,n_r)\mapsto \varrho_{r,k}(n_1,\ldots,n_r)$
is multiplicative, hence its Dirichlet series can be expanded into an
Euler product. Using \eqref{varrho_val} we deduce
\begin{equation*}
\sum_{n_1,\ldots,n_r=1}^{\infty}
\frac{\varrho_{r,k}(n_1,\ldots,n_r)}{n_1^{s_1}\cdots n_r^{s_r}}=
\prod_p \sum_{\nu_1,\ldots,\nu_r=0}^{\infty}
\frac{\varrho_{r,k}(p^{\nu_1},\ldots,p^{\nu_r})}{p^{\nu_1 s_1+\cdots
+\nu_r s_r}}
\end{equation*}
\begin{equation*}
= \prod_p \left(1+ \sum_{j=1}^{k-1} \sum_{1\le i_1<\ldots <i_j\le r}
\sum_{\nu_{i_1},\ldots,\nu_{i_j}=1}^{\infty}
\frac1{p^{\nu_{i_1}s_{i_1} +\cdots +\nu_{i_j}s_{i_j}}}\right)
\end{equation*}
\begin{equation*}
= \prod_p \left(1+ \sum_{j=1}^{k-1} \sum_{1\le i_1<\ldots <i_j\le r}
\frac1{p^{s_{i_1}}} \left(1-\frac1{p^{s_{i_1}}} \right)^{-1} \cdots
\frac1{p^{s_{i_j}}} \left(1-\frac1{p^{s_{i_j}}} \right)^{-1}\right)
\end{equation*}
\begin{equation*}
= \zeta(s_1)\cdots \zeta(s_r) \prod_p \left( \prod_{\ell=1}^r
\left(1-\frac1{p^{s_{\ell}}}\right) + \sum_{j=1}^{k-1} \sum_{1\le
i_1<\ldots <i_j\le r} \frac1{p^{s_{i_1}+\cdots + s_{i_j}}}
\prod_{\substack{\ell=1\\ \ell \ne i_1,\ldots,i_j}}^r
\left(1-\frac1{p^{s_{\ell}}} \right)\right)
\end{equation*}
\begin{equation*}
=\zeta(s_1)\cdots \zeta(s_r)\prod_p \left(1- \sum_{j=k}^r (-1)^{j-k}
\binom{j-1}{k-1} e_j(p^{s_1},\ldots,p^{s_r})\right),
\end{equation*}
by using Lemma \ref{lemma_pol_id} for $x_1=p^{s_1},\ldots,
x_r=p^{s_r}$ in the last step.
\end{proof}

\begin{proof}[Proof of Theorem {\rm \ref{Th_main}}] According to Theorem \ref{Th_Dir_ser}, for every $n_1,\ldots,n_r\in
\N$,
\begin{equation} \label{convo_id}
\varrho_{r,k}(n_1,\ldots,n_r) = \sum_{d_1\mid n_1, \ldots, d_r\mid
n_r} \psi_{r,k}(d_1,\ldots,d_r),
\end{equation}
where
\begin{equation*}
\sum_{n_1,\ldots,n_r=1}^{\infty}
\frac{\psi_{r,k}(n_1,\ldots,n_r)}{n_1^{s_1}\cdots n_r^{s_r}}=
D_{r,k}(s_1,\ldots,s_r).
\end{equation*}

The function $\psi_{r,k}$ is also multiplicative, symmetric in the
variables and for any prime powers $p^{\nu_1},\ldots,p^{\nu_r}$,
\begin{equation} \label{psi}
\psi_{r,k}(p^{\nu_1},\ldots,p^{\nu_r})=
\begin{cases} 1, & \nu_1=\ldots=\nu_r=0,\\
(-1)^{j-k+1}\binom{j-1}{k-1}, & \nu_1,\ldots,\nu_r\in \{0,1\}, \ j:=\nu_1+\ldots+\nu_r \ge k,\\
0, & \text{otherwise}.
\end{cases}
\end{equation}

Note that $\psi_{r,k}(p^{\nu_1},\ldots,p^{\nu_r})=0$ provided that
$\nu_i\ge 2$ for at least one $1\le i\le r$, or
$\nu_1,\ldots,\nu_r\in \{0,1\}$ and $\nu_1+\ldots+\nu_r <k$. For
$k\ge 2$ one has $\psi_{r,k}(p,1,\ldots,1)=0$ and for $k\ge 3$ one
has $\psi_{r,k}(p,p,1,\ldots,1)=0$, where $p$ is any prime.

From \eqref{convo_id} we deduce
\begin{equation*}
\sum_{n_1,\ldots,n_r\le x } \varrho_{r,k}(n_1,\ldots,n_r) =
\sum_{d_1, \ldots, d_r\le x} \psi_{r,k}(d_1,\ldots,d_r) \left\lfloor
\frac{x}{d_1}\right\rfloor \cdots \left\lfloor
\frac{x}{d_r}\right\rfloor
\end{equation*}
\begin{equation*}
= \sum_{d_1, \ldots, d_r\le x} \psi(d_1,\ldots,d_r)
\left(\frac{x}{d_1}+ O(1) \right) \cdots
\left(\frac{x}{d_r}+O(1)\right)
\end{equation*}
\begin{equation} \label{main_term}
= x^r \sum_{d_1, \ldots, d_r\le x}
\frac{\psi_{r,k}(d_1,\ldots,d_r)}{d_1\cdots d_r} + Q_{r,k}(x),
\end{equation}
with
\begin{equation*}
Q_{r,k}(x)\ll \sum_{u_1,\ldots,u_r} x^{u_1+\cdots+u_r}
\sum_{d_1,\ldots, d_r\le x}
\frac{|\psi_{r,k}(d_1,\ldots,d_r)|}{d_1^{u_1}\cdots d_r^{u_r}},
\end{equation*}
where the first sum is over $u_1,\ldots,u_r\in \{0, 1\}$ such that
at least one $u_i$ is $0$. Let $u_1,\ldots,u_r$ be fixed and assume
that $u_r=0$. Since $(x/d_i)^{u_i}\le x/d_i$ for every $i$, we have
\begin{equation*}
A:= x^{u_1+\cdots+u_r} \sum_{d_1,\ldots, d_r\le x}
\frac{|\psi_{r,k}(d_1,\ldots,d_r)|}{d_1^{u_1}\cdots d_r^{u_r}}\le
x^{r-1} \sum_{d_1,\ldots, d_r\le x}
\frac{|\psi_{r,k}(d_1,\ldots,d_r)|}{d_1\cdots d_{r-1}}
\end{equation*}

Assume that $k\ge 3$. Then
\begin{equation*}
A \le x^{r-1} \sum_{d_1,\ldots, d_r =1}^{\infty}
\frac{|\psi_{r,k}(d_1,\ldots,d_r)|}{d_1\cdots d_{r-1}}\ll x^{r-1},
\end{equation*}
since the series $D_{r,k}(1,\ldots,1,0)$ is absolutely convergent
for $k\ge 3$ by Theorem \ref{Th_Dir_ser}. We obtain that
\begin{equation} \label{Q_rk}
Q_{r,k}(x)\ll x^{r-1} \quad (k\ge 3).
\end{equation}

If $k=2$, then
\begin{equation} \label{estimate_A_Mertens}
A \le x^{r-1} \prod_{p\le x} \sum_{\nu_1,\ldots, \nu_r=0}^{\infty}
\frac{|\psi_{r,2}(p^{\nu_1},\ldots,p^{\nu_r})|}{p^{\nu_1+\cdots
+\nu_{r-1}}}
\end{equation}
\begin{equation*}
= x^{r-1} \prod_{p\le x} \left(1+\frac{r-1}{p}+
\frac{c_2}{p^2}+\cdots +\frac{c_{r-1}}{p^{r-1}}\right),
\end{equation*}
by \eqref{psi}, where $c_2,\ldots,c_{r-1}$ are certain positive
integers, using also that we have $p$ in the denominator if and only
if $\nu_r=1$ and exactly one of $\nu_1,\ldots,\nu_{r-1}$ is $1$, the
rest being $0$, which occurs $r-1$ times. We deduce that
\begin{equation*}
A\ll x^{r-1} \prod_{p\le x} \left(1+\frac1{p}\right)^{r-1}\ll
x^{r-1} (\log x)^{r-1}
\end{equation*}
by Mertens' theorem. This shows that
\begin{equation} \label{Q_r2}
Q_{r,2}(x)\ll x^{r-1}(\log x)^{r-1}.
\end{equation}

Furthermore, for the main term of \eqref{main_term} we have
\begin{equation*}
\sum_{d_1, \ldots, d_r\le x}
\frac{\psi_{r,k}(d_1,\ldots,d_r)}{d_1\cdots d_r}
\end{equation*}
\begin{equation} \label{roro}
= \sum_{d_1,\ldots,d_r=1}^{\infty}
\frac{\psi_{r,k}(d_1,\ldots,d_r)}{d_1 \cdots d_r} - \sum_{\emptyset
\ne I \subseteq \{1,\ldots,r\}} \sum_{\substack{d_i>x, \, i\in I\\
d_j\le x, \, j\notin I}} \frac{\psi_{r,k}(d_1,\ldots,d_r)}{d_1\cdots
d_r},
\end{equation}
where the series is convergent by Theorem \ref{Th_Dir_ser} and its
sum is $D_{r,k}(1,\ldots,1)=A_{r,k}$, given by \eqref{A_rk}.

Let $I$ be fixed and assume that $I=\{1,2,\ldots,t\}$, that is
$d_1,\ldots,d_t>x$ and $d_{t+1},\ldots,d_r\le x$, where $t\ge 1$. We
estimate the sum
\begin{equation*}
B:=\sum_{\substack{d_1,\ldots,d_t > x\\ d_{t+1},\ldots,d_r\le x}}
\frac{|\psi_{r,k}(d_1,\ldots,d_r)|}{d_1\cdots d_r}
\end{equation*}
by distinguishing the following cases:

Case i) $k\ge 3$, $t\ge 1$:
\begin{equation*}
B < \frac1{x} \sum_{d_1, \ldots, d_r=1}^{\infty}
\frac{|\psi_{r,k}(d_1,\ldots,d_r)|}{d_2\cdots d_r}\ll \frac1{x},
\end{equation*}
since the series is convergent by Theorem \ref{Th_Dir_ser}.

Case ii) $k=2$, $t\ge 3$: if $0<\varepsilon <1/2$, then
\begin{equation*}
B = \sum_{\substack{d_1,\ldots,d_t> x\\ d_{t+1},\ldots,d_r \le x}}
\frac{|\psi_{r,2}(d_1,\ldots,d_r)| d_1^{\varepsilon-1/2} \cdots
d_t^{\varepsilon-1/2}}{d_1^{1/2+\varepsilon}\cdots
d_t^{1/2+\varepsilon} d_{t+1}\cdots d_r}
\end{equation*}
\begin{equation*}
< x^{t(\varepsilon -1/2)} \sum_{d_1,\ldots,d_r=1}^{\infty}
\frac{|\psi_{r,2}(d_1,\ldots,d_r)|}{d_1^{1/2+\varepsilon}\cdots
d_t^{1/2+\varepsilon} d_{t+1}\cdots d_r} \ll x^{t(\varepsilon
-1/2)},
\end{equation*}
since the series is convergent (for $t\ge 1$). Using that
$t(\varepsilon -1/2)< -1$ for $0<\varepsilon< (t-2)/(2t)$, here we
need $t\ge 3$, we obtain $B\ll \frac1{x}$.

Case iii) $k=2$, $t=1$: Let $d_1>x$, $d_2,\ldots,d_r\le x$ and
consider a prime $p$.  If $p\mid d_i$ for an $i\in \{2,\ldots,r\}$,
then $p\le x$. If $p\mid d_1$ and $p>x$, then $p\nmid d_i$ for every
$i\in \{2,\ldots,r\}$ and $\psi_{r,2}(d_1,\ldots,d_r)=0$ by its
definition \eqref{psi}. Hence it is enough to consider the primes
$p\le x$. We deduce
\begin{equation*}
B < \frac1{x} \sum_{\substack{d_1> x\\ d_{2},\ldots,d_r\le x}}
\frac{|\psi_{r,2}(d_1,\ldots,d_r)|}{d_2\cdots d_r}
\end{equation*}
\begin{equation*}
\le \frac1{x} \prod_{p\le x} \sum_{\nu_1,\ldots, \nu_r=0}^{\infty}
\frac{|\psi_{r,2}(p^{\nu_1},\ldots,p^{\nu_r})|}{p^{\nu_2+\cdots
+\nu_{r}}}\ll \frac1{x}(\log x)^{r-1},
\end{equation*}
similar to the estimate of \eqref{estimate_A_Mertens}.

Case iv) $k=2$, $t=2$: We split the sum $B$ into two sums, namely
\begin{equation*}
B = \sum_{\substack{d_1> x,d_2>x\\ d_3,\ldots,d_r\le x}}
\frac{|\psi_{r,2}(d_1,\ldots,d_r)|}{d_1\cdots d_r}
\end{equation*}
\begin{equation*}
= \sum_{\substack{d_1> x^{3/2},d_2>x\\ d_3,\ldots,d_r\le x}}
\frac{|\psi_{r,2}(d_1,\ldots,d_r)|}{d_1\cdots d_r} +
\sum_{\substack{x^{3/2}\ge d_1> x,d_2>x\\ d_3,\ldots,d_r\le x}}
\frac{|\psi_{r,2}(d_1,\ldots,d_r)|}{d_1\cdots d_r}=: B_1+B_2,
\end{equation*}
say, where
\begin{equation*}
B_1= \sum_{\substack{d_1> x^{3/2},d_2>x\\ d_3,\ldots,d_r\le x}}
\frac{|\psi_{r,2}(d_1,\ldots,d_r)|}{d_1^{1/3}d_2\cdots d_r}
\frac1{d_1^{2/3}}
\end{equation*}
\begin{equation*}
<\frac1{x} \sum_{d_1,\ldots,d_r=1}^{\infty}
\frac{|\psi_{r,2}(d_1,\ldots,d_r)|}{d_1^{1/3}d_2\cdots d_r} \ll
\frac1{x},
\end{equation*}
since the series is convergent. Furthermore,
\begin{equation*}
B_2 <\frac1{x} \sum_{\substack{x^{3/2}\ge d_1, d_2>x \\
d_3,\ldots,d_r\le x}}
\frac{|\psi_{r,2}(d_1,\ldots,d_r)|}{d_1d_3\cdots d_r},
\end{equation*}
where $d_1\le x^{3/2}$, $d_2>x$, $d_3,\ldots,d_r\le x$.
Consider a prime $p$. If $p\mid d_i$ for an $i\in \{1,3,\ldots,r\}$,
then $p\le x^{3/2}$. If $p\mid d_2$ and $p>x^{3/2}$, then $p\nmid
d_i$ for every $i\in \{1,3,\ldots,r\}$ and
$\psi_{r,2}(d_1,\ldots,d_r)=0$ by its definition. Hence it is enough
to consider the primes $p\le x^{3/2}$. We deduce, cf. the estimate
of \eqref{estimate_A_Mertens},
\begin{equation*}
B_2 <\frac1{x} \prod_{p\le x^{3/2}} \sum_{\nu_1,\ldots,
\nu_r=0}^{\infty}
\frac{|\psi_{r,2}(p^{\nu_1},\ldots,p^{\nu_r})|}{p^{\nu_1+\nu_3+\cdots
+\nu_{r}}} \ll \frac1{x}(\log x^{3/2})^{r-1} \ll \frac1{x}(\log
x)^{r-1}.
\end{equation*}

Hence, given any $t\ge 1$, we have $B\ll \frac1{x}$ for $k\ge 3$ and
$B\ll \frac1{x}(\log x)^{r-1}$ for $k=2$. Therefore, by
\eqref{roro},
\begin{equation} \label{second}
\sum_{d_1, \ldots, d_r\le x}
\frac{\psi_{r,k}(d_1,\ldots,d_r)}{d_1\cdots d_r} = A_{r,k}+
O(R_{r,k}(x))
\end{equation}
with the notation \eqref{A_rk} and \eqref{R_rk}.

The proof is complete by putting  together \eqref{main_term},
\eqref{Q_rk}, \eqref{Q_r2} and \eqref{second}.
\end{proof}

\section{Acknowledgement} The author thanks the referee for careful reading and useful comments.

\medskip

\noindent L\'aszl\'o T\'oth \\
Department of Mathematics, University of P\'ecs \\ Ifj\'us\'ag
\'utja 6, H-7624 P\'ecs, Hungary \\ E-mail: {\tt
ltoth@gamma.ttk.pte.hu}

\end{document}